\documentclass[11pt]{article}

\usepackage[T1]{fontenc}
\usepackage{lmodern}
\usepackage[a4paper,margin=1.02in]{geometry}
\usepackage{amsmath,amssymb,amsthm,mathtools}
\usepackage{enumitem}
\usepackage{microtype}
\usepackage{xcolor}
\usepackage[colorlinks=true,linkcolor=blue!55!black,citecolor=blue!55!black,urlcolor=blue!55!black]{hyperref}
\allowdisplaybreaks
\numberwithin{equation}{section}
\newtheorem{theorem}{Theorem}[section]
\newtheorem{proposition}[theorem]{Proposition}
\newtheorem{lemma}[theorem]{Lemma}
\newtheorem{corollary}[theorem]{Corollary}
\newtheorem{definition}[theorem]{Definition}
\newtheorem{remark}[theorem]{Remark}
\newtheorem{conjecture}[theorem]{Conjecture}

\newcommand{\E}{\mathbb E}
\newcommand{\Prob}{\mathbb P}

\newcommand{\1}{\mathbf 1}

\newcommand{\Pmin}{P_{\min}}

\DeclareMathOperator{\dist}{dist}
\DeclareMathOperator{\Vol}{Vol}
\DeclareMathOperator{\Law}{Law}
\title{Pointwise subexponential growth\\
and near-diffusive displacement on bounded-degree graphs with non-negative Ollivier--Ricci curvature}
\author{Chiyu Zhou\thanks{School of Mathematical Sciences, University of Science and Technology of China, Hefei 230026, China. Email address: \texttt{dovong@mail.ustc.edu.cn}.}}
\date{July 2026}
\hypersetup{
 pdftitle={Pointwise subexponential growth and near-diffusive displacement on bounded-degree graphs with non-negative Ollivier--Ricci curvature},
 pdfauthor={Chiyu Zhou},
 pdfsubject={Log-Harnack inequalities, random-walk displacement, and Ollivier curvature},
 pdfkeywords={log-Harnack inequality, Ollivier curvature, quasi-doubling, random walk, relative entropy, Davies estimate}
}

\begin{document}
\maketitle

\begin{abstract}
 Let $G=(V,E)$ be a possibly infinite, locally-finite graph with non-negative Ollivier-Ricci curvature and with degrees bounded by $d<\infty$. We prove that there exists a constant $C_d$ such that the continuous-time random walk displacement and log-volume growth satisfy
\[
 \E_x\dist(x,X_t)^2\le t\exp\left[C_d\sqrt{\log (t)\log\log (t)}\right],
\]
\[
 \log\Vol(B(x,r))\le \exp\left[C_d\sqrt{\log (r)\log\log (r)}\right],
\]
for every $x\in V$ and $r,t\ge e^e$. 
\end{abstract}


\section{Introduction}
In Riemannian geometry,  a lower bound of Ricci curvature provides a series of estimates for the underlying manifold, including Bishop-Gromov volume comparison theorem, diameter bounds, splitting theorems, and spectral estimates~\cite{Jost}. These classical results
have motivated the development of discrete notions
of curvature
\cite{Michael,Linyau,Forman,Erbar,Ollivier2009,
jost2021characterizationsformancurvature,najman,generalgraph,
barton2026ollivierriccicurvaturecausalsets}.

Ollivier-Ricci curvature~\cite{Ollivier2009} formulates curvature on arbitrary metric spaces, hence in particular on graphs and Markov chain in a transportation-based criterion.  Its consequences
for graphs include concentration, Liouville
properties, and restrictions on expansion and mixing time
\cite{JoulinOllivier2010,M2023OR,MuenchSalez2022,Salez2022}.

We first introduce a general Markov-chain framework in which several notions used throughout the paper will be defined. Let \(P\) be an
\textit{irreducible Markov kernel} on a countable state space \(V\). We assume
that \(P\) is \emph{weakly reversible}, in the sense that
\[
    P(x,y)>0
    \quad\Longleftrightarrow\quad
    P(y,x)>0
    \qquad\text{for all }x,y\in V.
\]
When needed, we impose the stronger assumption that \(P\) is
\emph{reversible}: there exists a strictly positive measure
\(m:V\to(0,\infty)\) such that
\begin{equation*}
    m(x)P(x,y)=m(y)P(y,x),
    \qquad x,y\in V.
\end{equation*}
Such a measure \(m\) is called a \emph{reversible measure} for \(P\).

The support of \(P\) therefore determines an undirected graph on \(V\):
we write \(x\sim y\) whenever \(P(x,y)>0\). The \textit{distance} is defined via
\[
    \dist(x,y)
    :=
    \min\bigl\{k\geq0:P^k(x,y)>0\bigr\}.
\]

For probability measures \(\mu,\nu\) on \(V\),
their \textit{\(1\)-Wasserstein distance} is
\[
    W_1(\mu,\nu)
    :=
    \inf_{\pi}
    \sum_{u,v\in V}\dist(u,v)\pi(u,v),
\]
where the infimum is taken over all couplings \(\pi\) of \(\mu\) and
\(\nu\). Let
\[
    \widehat P:=\frac12(I+P)
\]
be the \(1/2\)-lazy version of \(P\). We say that \(P\) has
\emph{non-negative Ollivier--Ricci curvature} if
\[
    W_1\bigl(\widehat P(x,\cdot),\widehat P(y,\cdot)\bigr)
    \leq 1,
    \qquad\text{whenever }x\sim y.
\]
See
\cite{Salez2025modernaspectsmarkovchains} for equivalent definitions.

Some arguments developed later apply to
this general class of weakly reversible transition matrices. Our main
geometric results, however, concern normalized random walks on
unweighted graphs. Accordingly, unless explicitly stated otherwise, we
henceforth specialize to a connected, locally finite, simple graph
\(G=(V,E)\) and define the \textit{random walk kernel}
\[
    P(x,y)
    :=
    \frac{\1_{\{x\sim y\}}}{\deg(x)}.
\]
In this case, reversibility is automatic with $m(x):=\deg(x)$, the  distance
induced by \(P\) coincides with the usual graph distance, and we say graph $G$ has
non-negative Ollivier--Ricci curvature if
\[
    W_1\bigl(\widehat P(x,\cdot),\widehat P(y,\cdot)\bigr)
    \leq 1,
    \qquad\text{whenever }x\sim y.
\]

Non-negative Ollivier--Ricci curvature is  a  local
condition: for an edge \(x\sim y\), its value depends only on the
neighborhoods of \(x\) and \(y\), hence on the graph inside a ball of
radius \(2\). Our concern is how such local property constrains
the geometry and random walk at large scales. The analogous local-to-global
question is fundamental in Riemannian geometry. A model result is the
\textit{Bishop--Gromov volume comparison theorem}: if \(M\) is a complete
\(n\)-dimensional Riemannian manifold with \(\operatorname{Ric}\geq 0\), then
for every \(x\in M\) the function
\[
 r^{-n} \operatorname{Vol}(B(x,r))
\]
is non-increasing. In particular,
\(\operatorname{Vol}(B(x,r))\leq \omega_n r^n\), where $\omega_n$ is the volume of a unit ball in $\mathbb R^n$, and \(M\) has uniformly
 doubling volume: $\Vol(B(x,2r))/\Vol(B(x,r))\leq C$, for some $C=C(n)$. Another phenomenon is that the Brownian motion on manifold \(M\) with \(\operatorname{Ric}\geq 0\) is at most diffusive~\cite{hsu}:
$$ \E_x\dist(x,X_t)^2\le O(t).$$
We write $\mathbb{E}_x$ for expectations taken with respect to the law of the random walk $X$ started with $X_0 = x$.

To formulate the corresponding questions on reversible Markov kernel $P$, we first recall the
discrete- and continuous-time random walks on \(V\). 
The lazy version \(\widehat P\) induces a discrete-time random walk\footnote{We follow the convention in \cite{hutchcroft2025boundeddegreegraphsnonnegativeollivierricci}.}
\((X_n)_{n\geq0}\) on \(V\). If \(X_0\) has distribution \(\mu_0\), then
\[
    \mathbb P\bigl(X_0=x_0,\ldots,X_n=x_n\bigr)
    =
    \mu_0(x_0)
    \prod_{i=1}^{n}\widehat P(x_{i-1},x_i).
\]
In our paper, it is more convenient to work
with the continuous-time random walk \((X_t)_{t\geq0}\) (heat semigroup) having generator
\[
    L:=P-I=2(\widehat P-I).
\]
In particular,
\[
    \mathbb P_x(X_t=y)
    =
    P_t(x,y)
    =
    e^{-t}\sum_{k=0}^{\infty}
    \frac{t^k}{k!}P^k(x,y),
\]
and we can write as $e^{L}$.
We use the same notation \(X\) for both chains: an integer subscript
\(n\) refers to discrete time, while a real parameter \(t\) refers to
continuous time. 

For a reversible Markov kernel $P$ with its reversible measure $m$, we define the volume of  
\(A\subseteq V\)
\[
    \Vol(A):=\sum_{z\in A}m(z),
\]
and the graph distance ball of radius $r$ around $x$
\[
    B(x,r):=\{y\in V:\dist(x,y)\leq r\}.
\]

This naturally raises the discrete question of what
large-scale volume and displacement estimates can be recovered from
 non-negative Ollivier--Ricci curvature on graphs. Ollivier~\cite[Problem L]{Ollivier2010} asked whether there is any analogue of 
Bishop--Gromov volume comparison theorem under non-negative Ollivier--Ricci curvature. Salez~\cite{Salez2022} established that bounded-degree non-negatively Ollivier-curved graphs
cannot be \textit{expanders}. Hutchcroft and Lopez
\cite{hutchcroft2024relationisoperimetrytotalvariation} later obtained a quantitative form of Salez's theorem. These were important breakthroughs,
but remained too weak to imply a subexponential volume bound for bounded-degree non-negatively Ollivier-curved graphs. The authors of \cite{meansquaredisplacement} studied the relation between displacement and Ollivier--Ricci curvature. Salez~\cite{Salez2022} proved the first-moment estimate
\[
\mathbb{E}_{X_0\sim\pi}\left[d(X_0,X_n)\right]=o(n)
\]
for bounded-degree non-negatively Ollivier-curved graphs.

A major recent advance was obtained by Hutchcroft and M\"unch
\cite[Theorem~1.1]{hutchcroft2025boundeddegreegraphsnonnegativeollivierricci}. They proved averaged subexponential volume
growth and averaged near-diffusive displacement for finite
bounded-degree graphs of non-negative Ollivier--Ricci curvature, and
extended their results to unimodular random rooted graphs. In particular,
their theorem applies to infinite transitive graphs.

Since \(\Vol(A)\leq d|A|\) on graphs with degrees bounded by \(d\), their
cardinality estimate yields the following volume-form consequence after
adjusting the degree-dependent constant.
\begin{theorem}[Hutchcroft--M\"unch]
For each $d<\infty$, there exists a constant $C_d$ such that if $G = (V, E)$ is a finite, connected graph with non-negative Ollivier--Ricci curvature and
degrees bounded by $d$, then
\[
\frac{1}{|V|} \sum_{x\in V} \mathbf{E}_x\left[ d(X_0, X_n)^2 \right] 
\leq n\exp\left[ C_d\sqrt{\log n} \right]  
= n^{1+o(1)}\qquad \text{and}
\]
\[
 \frac{1}{|V|} \sum_{x\in V} \log \Vol(B(x,r))
 \leq \exp\left[ C_d\sqrt{\log r} \right] 
 = r^{o(1)} 
\]

for every $r, n \geq 2$.
\end{theorem}
Hutchcroft and M\"unch also conjectured that bounded-degree graphs of non-negative Ollivier--Ricci curvature have polynomial growth of bounded dimension and diffusive random walks, uniformly over the choice of root vertex~\cite[Conjecture~5.1]{hutchcroft2025boundeddegreegraphsnonnegativeollivierricci}:
\begin{conjecture}[Hutchcroft--M\"unch]
For each $d <\infty$, there exists a constant $C_d$ such that if $G = (V, E)$ is a graph with
non-negative Ollivier--Ricci curvature and degrees bounded by $d$ then
\[
\Vol(B(x,r)) \leq r^{C_d} \qquad \text{and} \qquad \mathbb{E}_x\left[d(x, X_n)^2\right] \leq C_d n
\]
for every $n, r \geq 1$ and $x \in V$.
\end{conjecture}

\section{Main results}

For $t\ge 0$ and $x\in V$, we define the heat entropy as  $$\mathcal H_t(x):=-\sum_{z\in V}P_t(x,z)\log \frac{P_t(x,z)}{m(z)}.$$

The main result of this paper provides a pointwise subexponential estimate of volume growth, heat entropy, and near diffusive displacement, uniformly of root vertex.
\begin{theorem}\label{thm:quasi-intro}
If  $G=(V,E)$  is a possibly infinite, connected graph with non-negative Ollivier-Ricci curvature and  degrees bounded by $d$ ,  then there exists a constant $C_d$ such that, uniformly in
$x\in V$ and $r,t\geq e^e$,
\begin{align*}
 \E_x\dist(x,X_t)^2&\le t\exp\left[C_d\sqrt{\log (t)\log\log (t)}\right],\\
\mathcal H_t(x)&\le \exp\left[C_d\sqrt{\log (t)\log\log (t)}\right],\label{eq:optimized-entropy-curvature}\\
 \log\Vol(B(x,r))&\le \exp\left[C_d\sqrt{\log (r)\log\log (r)}\right].
\end{align*}

\end{theorem}
Here we state the displacement estimate for the continuous-time random walk because, in our proof, it controls the heat entropy, which in turn controls volume growth. We also obtain a pointwise near-diffusive displacement estimate for the discrete-time random walk.

\begin{corollary}\label{cor:discretedisplacement}
If  $G=(V,E)$  is a possibly infinite, connected graph with non-negative Ollivier-Ricci curvature and  degrees bounded by $d$ ,  then there exists a constant $C_d$ such that, uniformly in
$x\in V$ and $n\geq e^e$,
\[
 \E_x\dist(x,X_n)^2\le n\exp\left[C_d\sqrt{\log (n)\log\log (n)}\right].
\]
\end{corollary}

In fact, we prove the above property under the following log-harnack inequality and prove that Ollivier-Ricci curvature supplies log-Harnack. For probability measures \(\mu,\nu\) on the same measurable space, their
relative entropy (Kullback--Leibler divergence; see~\cite{kullbackleibler}) is
\[
        D(\mu\|\nu)
        :=
        \begin{cases}
        \displaystyle\int\log\!\left(\frac{d\mu}{d\nu}\right)d\mu,
            &\mu\ll\nu,\\[2mm]
        +\infty,&\text{otherwise}.
        \end{cases}
\]

\begin{definition}\label{def:LH}
For $A<\infty$, we say that a Markov kernel $P$ satisfies $\mathrm{LH}(A)$ if its associated heat semigroup $P_t$ satisfies
\begin{equation*}
 D(P_t(x,\cdot)\|P_t(y,\cdot))
 \le A\left(\frac{\dist(x,y)^2}{t}
 +\frac{\dist(x,y)}{\sqrt t}\right)
\end{equation*}
for every $x,y\in V$ and $t>0$.
\end{definition}

The condition \(\mathrm{LH}(A)\) turns out to be well suited to this
problem. Our argument is based on a self-improving loop:
\[
    \text{displacement}
    \Longrightarrow
    \text{entropy}
    \Longrightarrow
    \text{volume growth}
    \Longrightarrow
    \text{improved displacement}.
\]
At each iteration, the power exponent improves, while the corresponding
constant increases. Optimizing the number of iterations against this
growth of the constants yields the following subexponential estimates.
\begin{theorem}\label{thm:quasi-intro-lh}
If  $G=(V,E)$  is a possibly infinite, connected graph and its random walk kernel $P$ satisfies $\mathrm{LH}(A)$. Then, there exists a constant $C_A$ such that, uniformly in
$x\in V$  and $r,t\geq e^e$,
\begin{align*}
 \E_x\dist(x,X_t)^2&\le t\exp\left[C_A\sqrt{\log (t)\log\log (t)}\right],\\
\mathcal H_t(x)&\le \exp\left[C_A\sqrt{\log (t)\log\log (t)}\right],\label{eq:optimized-entropy}\\
 \log\Vol(B(x,r))&\le \exp\left[C_A\sqrt{\log (r)\log\log (r)}\right].
\end{align*}
\end{theorem}
Following this thread, we prove that diffusive displacement implies volume doubling under $\mathrm{LH}(A)$. 

\begin{theorem}[Diffusive displacement implies doubling]
\label{prop:retention-doubling}
If  a reversible Markov chain $P$  satisfies $\mathrm{LH}(A)$ and for some $\alpha>0$ and $K<\infty$,
\begin{equation}\label{eq:AlphaDisplacement}
    \sup_x\E_x\dist(x,X_t)^\alpha\le Kt^{\alpha/2},\qquad t\ge1,
\end{equation}
then  $$\Vol(B(x,2r))\le C_V\Vol(B(x,r)),\quad x\in V,\ r\ge 0$$ holds,
where $C_V$ depends only on $A,K,\alpha$.
\end{theorem}

Log-harnack inequalities are widely researched on manifolds. 
  Dimension-free and logarithmic Harnack
inequalities originate in the work of Wang and were subsequently developed
through coupling methods
\cite{Wang1997,ArnaudonThalmaierWang2006,RocknerWang2010,BakryGentilLedoux2015}.

\begin{theorem}\label{thm:entropy-intro}
If  $G=(V,E)$  is a possibly infinite, connected graph with non-negative Ollivier-Ricci curvature and  degrees bounded by $d$, then $G$ satisfies $\mathrm{LH}(32d)$, that is,
\begin{equation*}
 D(P_t(x,\cdot)\|P_t(y,\cdot))
 \le 32d\left(\frac{\dist(x,y)^2}{t}
 +\frac{\dist(x,y)}{\sqrt t}\right).
\end{equation*}
Equivalently, for bounded $f>0$,
\begin{equation*}
     P_t\log f(x)\le\log P_tf(y)
 +32d\left(\frac{\dist(x,y)^2}{t}
 +\frac{\dist(x,y)}{\sqrt t}\right).
\end{equation*}
\end{theorem}

\medskip
\noindent\textbf{Proof architecture.}
The paper is organized around a primitive log-Harnack inequality rather
than a curvature condition. The manuscript contains exactly three substantive implications:
\begin{equation*}\label{eq:dependency-map}
\begin{gathered}
 \mathrm{LH}(A)\Longrightarrow
 \left\{\begin{array}{c}
 \text{pointwise near-diffusive displacement},\\[-1mm]
 \text{subexponential volume growth};
 \end{array}\right.\\[1mm]
 \mathrm{LH}(A)+\text{continuous diffusive displacement}
 \Longrightarrow \text{Volume Doubling};\\[1mm]
 \text{nonnegative Ollivier curvature with }\deg\le d
 \Longrightarrow \mathrm{LH}(Cd).
\end{gathered}
\end{equation*}
The first are proved by the bootstrap of displacement-entropy-volume loop  and the second is an application of  the method we use in the first part; the third is proved by Wang's Coupling Method.

\section{Proof of Theorem \ref{thm:quasi-intro-lh}}
Throughout this section, we assume  $G=(V,E)$ be a connected, locally finite  graph, let $P$ be its random-walk kernel, and take $m(x)=\deg(x)$. 

First, we introduce some basic propositions of relative entropy and log-Harnack inequality $LH(A)$.

\begin{proposition}[Data processing]\label{prop:data-processing}
Let \(\mu,\nu\) be probability measures and let \(T\) be measurable.
Then
\begin{equation*}
        D(T_\#\mu\|T_\#\nu)\le D(\mu\|\nu).
\end{equation*}
\end{proposition}

A more general data-processing inequality for $\alpha$-R\'enyi divergences is proved in \cite[Theorem 9]{vanErvenHarremoes2014}; the Kullback–Leibler case may also be obtained by taking the limit  $\alpha_n\to 1$. We include a direct proof for completeness. 

\begin{proof}
It is enough to consider \(D(\mu\|\nu)<\infty\).  Set
\(f=d\mu/d\nu\), let \(h=dT_\#\mu/dT_\#\nu\), and put
\(\mathcal G=\sigma(T)\).  First, \(T_\#\mu\ll T_\#\nu\).  For every
bounded measurable \(k\) on the target space,
\[
        \int k(T)f\,d\nu
        =\int k\,dT_\#\mu
        =\int kh\,dT_\#\nu
        =\int k(T)h(T)\,d\nu.
\]
Thus
\[
        h(T)=\E_\nu[f\mid\mathcal G]
        \quad\nu\text{-almost everywhere}.
\]
For \(\varphi(u)=u\log u\), conditional Jensen gives
\[
\begin{aligned}
 D(T_\#\mu\|T_\#\nu)
 &=\int\varphi(h(T))\,d\nu
  =\int\varphi\!\left(\E_\nu[f\mid\mathcal G]\right)d\nu\\
 &\le\int\varphi(f)\,d\nu
  =D(\mu\|\nu).
\end{aligned}
\]
\end{proof}
It is interesting to note that $\mathrm{LH}(A)$ gives a upper bound of  degree of $G$.
\begin{proposition}
\label{prop:LH-degree}
Assume $P$ is a reversible Markov chain satisfying $\mathrm{LH}(A)$, then
\begin{equation}\label{eq:LH-degree-bound}
 \deg(x)\le d_A:=\exp\{e(2A+\log2)\}, \text{ for every } x\in V.
\end{equation}
and \begin{equation}\label{eq:ratio}
    \frac{m(x)}{m(y)}\leq d_A, \text{ for every } x\sim y.
\end{equation}
\end{proposition}

\begin{proof}
Fix an edge $x\sim y$ and apply  Proposition \ref{prop:data-processing} to
the measurable map \(z\mapsto\1_{\{y\}}(z)\).
We have
\[
 D(P_1(y,\cdot)\|P_1(x,\cdot))\geq D((\1_{\{y\}})_\#P_1(y,\cdot)\|(\1_{\{y\}})_\#P_1(x,\cdot))
 \ge P_1(y,y)\log\frac1{P_1(x,y)}-\log2.
\]
Since $P_1(y,y)\ge e^{-1}$ and $\dist(x,y)=1$, the log-Harnack inequality at
time $1$ gives
\[
 2A\ge D(P_1(y,\cdot)\|P_1(x,\cdot))
 \ge e^{-1}\log\frac1{P_1(x,y)}-\log2.
\]
Thus,
$$P_1(x,y)\ge\exp\{-e(2A+\log2)\}.$$Therefore,
\[
 1\ge\sum_{y\sim x}P_1(x,y)\ge \exp\{-e(2A+\log2)\}\deg(x),
\]
which proves \eqref{eq:LH-degree-bound}.  
Moreover,$$\frac{m(x)}{m(y)}\leq \frac{P_1(y,x)}{P_1(x,y)}\leq d_A.$$ It is \eqref{eq:ratio}.
\end{proof}

\subsection*{Heat entropy}
Given a locally-finite graph $G=(V,E)$ and the heat semigroup $P_t$ associated to the random walk kernel $P$,
by detailed-balance,
\begin{equation*}
    m(x)P_t(x,y)
    =
    m(y)P_t(y,x),
    \qquad x,y\in V,\quad t\geq0.
\end{equation*}
Define the heat kernel with respect to \(m\) by
\begin{equation*}
    h_t^x(y)
    :=
    \frac{P_t(x,y)}{m(y)}.
\end{equation*}
Therefore, we write

\begin{equation*}
\mathcal H_t(x)=-\sum_{z\in V}P_t(x,z)\log h_t^x(z).
\end{equation*}
and define
\begin{equation*}
 \mathcal H^*(t):=\sup_{x\in V}\mathcal H_t(x).
\end{equation*}
and Shannon entropy
$$H_t(x):=-\sum_{z\in V}P_t(x,z)\log P_t(x,z).$$
It is direct to check that $m(x)P_t(x,y)=m(y)P_t(y,x)$.

Since $m(z)\ge1$ and $P_t(x,z)\le1$, one has $0<h_t^x(z)\le1$ whenever
$P_t(x,z)>0$, and hence $\mathcal H_t(x)\ge0$.
The monotonicity and linear upper bound of Shannon entropy is classic. The proof of them in $\mathcal H_t(x)$ are similar and we present it for completeness.
\begin{lemma}
\label{lem:entropy-linear}
Let $G=(V,E)$ be a graph satisfying $\mathrm{LH}(A)$. There is $C_A<\infty$ such that
\begin{equation}\label{eq:entropy-linear}
 \mathcal H^*(t)\le C_A(1+t),\qquad t\ge0.
\end{equation}
Moreover, $t\mapsto\mathcal H_t(x)$ is nondecreasing for each $x$.
\end{lemma}

\begin{proof}
By Proposition \ref{prop:LH-degree}, $G$ has bounded degree $d_A$. Therefore,
\[
 \mathcal H_t(x)=H_t(x)+\E_x\log m(X_t)
 \le H_t(x)+\log d_A,
\]
and Shannon entropy $H_t(x)\leq C_A(1+t)$. This proves \eqref{eq:entropy-linear}.

For monotonicity, 
reversibility gives \(h_{t+s}^x=P_sh_t^x\).  Since \(u\mapsto u\log u\)
is convex and \(m\) is invariant, Jensen's inequality gives
\[
        \int (P_sh_t^x)\log(P_sh_t^x)\,dm
        \le\int P_s(h_t^x\log h_t^x)\,dm
        =\int h_t^x\log h_t^x\,dm.
\]
\end{proof}
\subsection*{Displacement controls entropy}
The following lemma is the standard mutual-information identity for Shannon entropy; see, for example, \cite[Eq.~(2.40)]{CoverThomas2006}. We state its equivalent form for $\mathcal H_t(x)$.
\begin{lemma}
\label{lem:mutual-information}
Fix $x\in V$ and $s,t>0$. 
Then
\begin{equation}\label{eq:mutual-information}
 \mathcal H_{s+t}(x)-\E_x\mathcal H_t(X_s)= \sum_yP_s(x,y)
 D\bigl(P_t(y,\cdot)\|P_{s+t}(x,\cdot)\bigr).
\end{equation}

\end{lemma}

\begin{proof}
It is direct to check that
\begin{align*}
\mathcal H_{s+t}(x)-\E_x\mathcal H_t(X_s)&=-\sum_{z\in V}P_{s+t}(x,z)\log \frac{P_{s+t}(x,z)}{m(z)}+\sum_{y,z\in V}P_s(x,y)P_t(y,z)\log\frac{P_t(y,z)}{m(z)}\\
&=\sum_{y,z\in V}P_s(x,y)P_t(y,z)\log\frac{P_t(y,z)}{P_{s+t}(x,z)}\\
&=\sum_yP_s(x,y)
 D\bigl(P_t(y,\cdot)\|P_{s+t}(x,\cdot)\bigr).
\end{align*}
\end{proof}

\begin{lemma}
\label{lem:entropy-increment-sharp}
Let  $G=(V,E)$ be a graph satisfying $\mathrm{LH}(A)$. Suppose that, for some $\beta\in(1,2]$ and $M\ge1$,
\begin{equation}\label{eq:moment-beta}
 \sup_x\E_x\dist(x,X_s)^2\le M s^{\beta},\qquad s\ge1.
\end{equation}
Then, for $s,t\ge1$,
\begin{equation}\label{eq:entropy-increment-sharp}
 \mathcal H^*(s+t)\le\mathcal H^*(t)
 +4A\left(\frac{M s^{\beta}}{t}
 +\frac{\sqrt M\,s^{\beta/2}}{\sqrt t}\right).
\end{equation}
\end{lemma}

\begin{proof}
Let $Y'$ be an independent copy of $Y=X_s$.  Since
$P_{s+t}(x,\cdot)=\sum_wP_s(x,w)P_t(w,\cdot)$, convexity of relative
entropy gives
\begin{align*}
    \sum_yP_s(x,y)
 D\bigl(P_t(y,\cdot)\|P_{s+t}(x,\cdot)\bigr)&\le \sum_{y,z}P_s(x,y)P_s(x,z)
 D\bigl(P_t(y,\cdot)\|P_{t}(z,\cdot)\bigr)\\
    &=\E D(P_t(Y,\cdot)\|P_t(Y',\cdot)).
\end{align*}
 
The triangle inequality and \eqref{eq:moment-beta} give
\[
 \E\dist(Y,Y')^2\le4Ms^{\beta},\qquad
 \E\dist(Y,Y')\le2\sqrt M\,s^{\beta/2}.
\]
Apply $\mathrm{LH}(A)$ and then we have \eqref{eq:entropy-increment-sharp}.
\end{proof}

\begin{lemma}
\label{lem:scaled-iteration}
Let $a\in(0,1)$, $M\ge1$, and let $F$ be nondecreasing.  Suppose
\[
 F\bigl(t+M^{-a}t^a\bigr)\le F(t)+C_0,
 \qquad t\ge M.
\]
Then, with $\eta=1-a$,
\begin{equation*}
 F(t)\le F(M)+\frac{4C_0}{\eta}M^at^\eta,
 \qquad t\ge M.
\end{equation*}
\end{lemma}

\begin{proof}
Set $t_0=M$ and $t_{k+1}=t_k+M^{-a}t_k^a$.  Since $t_k\ge M$,
one has $t_{k+1}\le2t_k$.  Therefore
\[
 t_{k+1}^\eta-t_k^\eta
 =\eta\int_{t_k}^{t_{k+1}}u^{-a}\,du
 \ge \frac{\eta}{2}M^{-a}.
\]
Thus at most $2\eta^{-1}M^at^\eta$ steps are needed to pass $t$.
Iteration and monotonicity prove the claim.
\end{proof}

\begin{theorem}[Displacement controls entropy]
\label{prop:moment-entropy-sharp}
Let  $G=(V,E)$ be a graph satisfying $\mathrm{LH}(A)$.  Suppose that, for some $\beta\in(1,2]$ and $M\ge1$,
\begin{equation*}
 \sup_x\E_x\dist(x,X_s)^2\le M s^{\beta},\qquad s\ge1.
\end{equation*}
Put  $a=\frac1{\beta} \text{ and }\eta=1-{\frac1{\beta}}$.
Then, there exists a constant $C_E(A)$, such that
\begin{equation}\label{eq:entropy-sharp}
 \mathcal H^*(t)\le
 \frac{C_E(A)}{\eta}M^a t^{\eta},
 \qquad t\ge1.
\end{equation}
\end{theorem}

\begin{proof}
If $1\le t\le M$, Lemma~\ref{lem:entropy-linear} gives
$\mathcal H^*(t)\le C_At$.  Since $$t=t^{a} t^{\eta}\le M^{a} t^{\eta},$$ this is
bounded by the right-hand side of \eqref{eq:entropy-sharp}.

Now let $t\ge M$ and choose $s=(t/M)^{a}\geq 1$.  Then Lemma \ref{lem:entropy-increment-sharp} gives $$\mathcal H^*(t+(t/M)^{a})\leq \mathcal{H}^*(t)+8A.$$  Apply Lemma \ref{lem:entropy-increment-sharp} and Lemma~\ref{lem:scaled-iteration}, starting $t$ at $M$,
and use $\mathcal H^*(M)\le C_AM=C_AM^{a}M^{\eta}$.
\end{proof}

\subsection*{Entropy controls volume}

\begin{proposition}[Entropy--volume inequality]
\label{prop:entropy-volume}
Let  $G=(V,E)$ be a graph satisfying $\mathrm{LH}(A)$.  Then, for every $x\in V$, $R\ge0$, and $t>0$,
\begin{equation}\label{eq:entropy-volume}
 \log\Vol(B(x,R))\le\mathcal H_t(x)
 +A\left(\frac{R^2}{t}+\frac{R}{\sqrt t}\right).
\end{equation}
\end{proposition}

\begin{proof}
Let $B=B(x,R)$ and let
$\nu_B(\cdot):=m(\cdot)\1_B(\cdot)/\Vol(B)$.  Define the modified probability measure
$$\overline P_t(\cdot):=\sum_{y\in B}\nu_B(y)P_t(y,\cdot)$$Reversibility gives
\begin{equation}\label{eq:upperoverlinept}
 \overline P_t(z)=\frac{m(z)}{\Vol(B)}P_t(z,B)
 \le\frac{m(z)}{\Vol(B)}.
\end{equation}
Convexity and $\mathrm{LH}(A)$ give
\begin{align*}
    D(P_t(x,\cdot)\|\overline P_t(\cdot))&=D(P_t(x,\cdot)\|\sum_{y\in B}\nu_B(y)P_t(y,\cdot))\\
&\le\sum_{y\in B}\nu_B(y)D(P_t(x,\cdot)\|P_t(y,\cdot))
\\&\le A\left(\frac{R^2}{t}+\frac{R}{\sqrt t}\right).
\end{align*}
 
The upper bound \eqref{eq:upperoverlinept} on $\overline P_t$ gives the reverse estimate
$$D(P_t(x,\cdot)\|\overline P_t(\cdot))\ge\log\Vol(B)-\mathcal H_t(x).$$We prove \eqref{eq:entropy-volume}.
\end{proof}

\begin{theorem}[Entropy controls volume]
\label{cor:entropy-volume-sharp}
Let  $G=(V,E)$ be a graph satisfying $\mathrm{LH}(A)$.  Suppose that
$$\mathcal H^*(t)\le Ht^\eta$$ for $t\ge1$, where $H\ge1$ and
$\eta\in(0,1)$. 
Then, with $\gamma:=\frac{2\eta}{1+\eta}$,
\begin{equation*}
 \log\Vol(B(x,R))\le C_V(A)H^{1/(1+\eta)}R^\gamma,
 \qquad R\ge1.
\end{equation*}
\end{theorem}

\begin{proof}
If $R^2\ge H$, choose
$t=(R^2/H)^{1/(1+\eta)}$.  Then the entropy term and $R^2/t$ both equal
$H^{1/(1+\eta)}R^\gamma$, while the term $R/\sqrt t\leq 1+R^2/t$.  If $R^2<H$, use the Moore bound $\log\Vol(B(x,R))\leq C_A R$;  the inequality
$R^{1-\gamma}\le H^{1/(1+\eta)}$ in this range absorbs it.
\end{proof}

\subsection*{Volume controls displacement}
For $t>0$ and $r\ge0$, define
\begin{equation*}
 \zeta_t(r):=r\operatorname{arsinh}(r/t)-\sqrt{t^2+r^2}+t.
\end{equation*}
Let $G=(V,E)$ be a graph. Recall that the heat semigroup is defined via $P_t=e^{t(P-I)}$ with random walk kernel $P$ and we write $h_t^x(y):=P_t(x,y)/m(y)$.

\begin{lemma}[Davies-Gaffney-Grigor’yan Lemma]
\label{lem:davies}
Let $G=(V,E)$ be a graph. For all $x,y\in V$ and $t>0$,
\begin{equation*}
 h_t^x(y)\le\frac1{\sqrt{m(x)m(y)}}
 \exp\{-\zeta_t(\dist(x,y))\}.
\end{equation*}
Moreover, for a universal $c_0\in (0,1)$,
\begin{equation}\label{eq:davies-zeta-lower}
 \zeta_t(r)\ge c_0
 \begin{cases}
 r^2/t,&0\le r\le t,\\
 r\log(1+r/t),&r>t.
 \end{cases}
\end{equation}
\end{lemma}

\begin{proof}
It is the heat-kernel estimate in \cite[Eq.~(3)]{BauerHuaYau2017}; see also \cite{BHY2015} and \cite{Delmotte99}. The function $\zeta_t(r)$ is defined via the Legendre transform and appears naturally on graphs; see \cite{Davies1993,Delmotte99,Pang93} for examples. Writing $\zeta_t(r)=tg(r/t)$ and estimating
$g(u)=u\operatorname{arsinh}u-\sqrt{1+u^2}+1$ in the ranges
$u\le1$ and $u\ge1$ proves \eqref{eq:davies-zeta-lower}.
\end{proof}

\begin{theorem}[Volume controls displacement]
\label{prop:volume-moment-sharp}
Let  $G=(V,E)$ be a graph satisfying $\mathrm{LH}(A)$.  Suppose that, for some $B\ge1$ and $0<\gamma\le2/3$,
\begin{equation}\label{eq:subexp-volume}
 \log\Vol(B(x,r))\le Br^\gamma,
 \qquad x\in V,\ r\ge1.
\end{equation}
 Then
\begin{equation}\label{eq:volume-moment-sharp}
 \sup_x\E_x\dist(x,X_t)^2
 \le C_D(A)(1+B)^{2/(2-\gamma)}t^{2/(2-\gamma)},
 \qquad t\ge1.
\end{equation}
\end{theorem}

\begin{proof}
Choose
\[
 L=[C_*(1+B)]^{1/{(2-\gamma)}},\qquad
 \rho=Lt^{1/(2-\gamma)},\qquad
 T=L^{(2-\gamma)/(1-\gamma)},
\]
where $C_*:=\frac{2^{5/3}}{c_0\log2}$ and $c_0$ in the equation \eqref{eq:davies-zeta-lower} .  Notice that $L\ge1$.
If $1\le t<T$, then
\[
 \E_x\dist(x,X_t)^2\le t^2+t\le2t^2
 \le2L^2t^{2/(2-\gamma)}.
\]
The last inequality follows from
$$t^{2(1-\frac{1}{2-\gamma})}<T^{2(1-\frac{1}{2-\gamma})}=L^2.$$

Suppose now that $t\ge T$, so that $\rho\le t$.  Put
$A_0=B(x,\rho)$ and, for $k\ge1$,
\[
 A_k=B(x,(k+1)\rho)\setminus B(x,k\rho).
\]
For $y\in A_k$, reversibility and the Davies-Gaffney-Grigor’yan Lemma give
\[
 P_t(x,y)=m(y)h_t^x(y)
 \le\sqrt{\frac{m(y)}{m(x)}}e^{-\zeta_t(\dist(x,y))}
 \le\sqrt{d_A}\,e^{-\zeta_t(k\rho)}.
\]
Since counting measure is bounded by degree volume,
\begin{equation}\label{eq:annulus-Davies-volume}
 \Prob_x(X_t\in A_k)
 \le \sqrt{d_A}\exp\{B((k+1)\rho)^\gamma-\zeta_t(k\rho)\}.
\end{equation}

First assume $k\rho\le t$.  Then
$\zeta_t(k\rho)\ge c_0 k^2\rho^2/t$.  For $k\ge1$,
$(k+1)^\gamma\le2^\gamma k^\gamma\le2^\gamma k^2$, and hence
\[
 B((k+1)\rho)^\gamma-c_0\frac{k^2\rho^2}{t}
 \le k^2t^{\gamma/(2-\gamma)}
       \{2^\gamma BL^\gamma-c_0L^2\}.
\]
Because $L^{2-\gamma}=C_*(1+B)$,$$2^\gamma BL^\gamma-c_0L^2\leq  \frac{2^\gamma BL^2}{C_*(1+B)}-c_0L^2\leq -\frac{c_0}{2} L^2.$$   Thus
\begin{equation}\label{eq:near-annulus-tail}
 \Prob_x(X_t\in A_k)
 \le \sqrt{d_A}\exp\{-\frac{c_0}{2} k^2L^2t^{\gamma/(2-\gamma)}\}\leq\sqrt{d_A}\exp\{-\frac{c_0}2 k^2 \}.
\end{equation}

Next assume $k\rho>t$ and put $s=k\rho$.  Since
\[
 T^{1-\gamma}=L^{2-\gamma}=C_*(1+B),
\]
we have $s^{1-\gamma}\ge C_*(1+B)$ and 
$\log(1+s/t)\ge\log2$.  Thus, 
\[
 B(2s)^\gamma\leq B2^\gamma \frac{s}{C_*(1+B)}\le\frac{c_0}2 s\log(1+s/t).
\]
As $(k+1)\rho\le2k\rho=2s$, equation
\eqref{eq:annulus-Davies-volume} and the far part of
\eqref{eq:davies-zeta-lower} yield
\begin{equation}\label{eq:far-annulus-tail}
 \Prob_x(X_t\in A_k)
 \le \sqrt{d_A}\exp\{-\frac{c_0}{2} s\log(1+s/t)\}
 \le \sqrt{d_A}e^{-\frac{c_0 \log 2}{2} k}.
\end{equation}

Finally,
\[
\begin{aligned}
 \E_x\dist(x,X_t)^2
 &\le \rho^2+
 \sum_{k\ge1}(k+1)^2\rho^2\Prob_x(X_t\in A_k)\\
 &\le \rho^2(1+\sqrt{d_A}\sum_{k\ge 1}(k+1)^2(e^{-\frac{c_0\log 2}2 k}+e^{-\frac{c_0}2 k^2}))\\
 &\le C'(A)\rho^2
 =C'(A)L^2t^{2/(2-\gamma)},
\end{aligned}
\]
where  $C'(A)$ is a constant only depending on $A$.

The near sum is controlled by \eqref{eq:near-annulus-tail}, and the far
sum by \eqref{eq:far-annulus-tail}; both are uniform because $L,\rho\ge1$.
Since $L^2=[C_*(1+B)]^{2/(2-\gamma)}$, this is
\eqref{eq:volume-moment-sharp}.
\end{proof}

\subsection*{Bootstrap to Subexponential Bounds }

\begin{theorem}[Entropy--volume--displacement loop]
\label{thm:finite-stage}
Let  $G=(V,E)$ be a graph satisfying $\mathrm{LH}(A)$. For any integer
$n\geq 0$, there are constants $M_n,H_n,B_n\ge1$, depending only on $A$ and $n$, such that, for all $x\in V$ and $t,R\ge1$,
\begin{align*}
 \E_x\dist(x,X_t)^2
 &\le M_n t^{1+1/(n+1)},\\
 \mathcal H^*(t)
 &\le H_n t^{1/(n+2)},\\
 \log\Vol(B(x,R))
 &\le B_nR^{2/(n+3)}.
\end{align*}

Then one may choose the stage constants so that
\begin{equation}\label{eq:constant-stage}
 \log(1+M_n)+\log(1+H_n)+\log(1+B_n)
 \le C(n+2)\log(e+n),
\end{equation}
where $C=C(A)$ is universal constant.
\end{theorem}

\begin{proof}
Now, we are in the right place to build the Entropy--volume--displacement loop. The Poisson construction gives
$\E_x\dist(x,X_t)^2\le t^2+t\le2t^2$, so take $M_0=2$. 

Applying Theorem \ref{prop:moment-entropy-sharp}, Theorem \ref{cor:entropy-volume-sharp}, and Theorem \ref{prop:volume-moment-sharp} sequentially, it gives a loop. Continue the looping.

We write, in the $n$-loop,  there are constants $M_n,H_n,B_n\ge1$ and $\beta_n,\eta_n,\gamma_n>0$, such that,
for all $x\in V$ and $t,r\ge1$,
\begin{align*}
 \E_x\dist(x,X_t)^2
 &\le M_n t^{\beta_n}\\
 \mathcal H^*(t)
 &\le H_n t^{\eta_n}\\
 \log\Vol(B(x,r))
 &\le B_nr^{\gamma_n}.
\end{align*}

Then, starting from $\beta_0=2$ at $0$-loop, we have
\[
\beta_n=\frac{n+2}{n+1},\qquad
 \eta_n=\frac1{n+2},\qquad
 \gamma_n=\frac2{n+3}.
\]
In the following, we trace the constants $M_n,H_n,B_n\ge1$. Theorem~~\ref{prop:moment-entropy-sharp} gives
\begin{equation}\label{eq:H-stage}
 H_n=\frac{C_{\mathrm E}(A)}{\eta_n}(1+M_n)^{1/{\beta_n}}
 =C_{\mathrm E}(A)(n+2)(1+M_n)^{(n+1)/(n+2)},
\end{equation}
Theorem~\ref{cor:entropy-volume-sharp} gives
\begin{equation}\label{eq:B-stage}
 B_n\le C_{\mathrm V}(A)H_n^{1/(1+\eta_n)} \leq C_V(A)\left[C_E(A)(n+2)(1+M_n)^{(n+1)/(n+2)}\right]^{(n+2)/(n+3)}.
\end{equation}
Theorem~\ref{prop:volume-moment-sharp} then gives
\begin{align*}
     M_{n+1}&=C_{D}(A)(1+B_n)^{2/(2-\gamma_n)}\\
     &= C_{D}(A)(1+B_n)^{(n+3)/(n+2)}\\
     & \leq C_E(A)C_D(A)(1+C_V(A))^{3/2}(n+2)(M_n+1)^{(n+1)/(n+2)}
\end{align*}
Pick $K_A:= C_E(A)C_D(A)(1+C_V(A))^{3/2}+1$, we have
\begin{equation}\label{M-stage}
    M_{n+1}+1\leq K_A(n+2)(M_n+1)^{(n+1)/(n+2)}.
\end{equation}
It gives
\[
\log(1+M_{n+1})\le\frac{n+1}{n+2}\log(1+M_n)+\log K_A+\log(n+2).
\]
  Therefore
$ \log(1+M_n)
 \le C (n+2)\log(e+n),$ for some sufficiently large constant $C=C(A)$.
 Equations
\eqref{eq:H-stage} and \eqref{eq:B-stage} give the same logarithmic bound
for $H_n$ and $B_n$.
\end{proof}

\begin{proof}[Proof of Theorem \ref{thm:quasi-intro-lh}]
From Theorem \ref{thm:finite-stage}, we have
\begin{equation}\label{eq:FinalDislacement}
 \log\left(\frac{\E_x\dist(x,X_t)^2}{t}\right)
 \le C(A)n\log(e+n)+\frac{\log t}{n+1},
\end{equation}
\begin{equation}\label{eq:FinalEntropy}
  \log \mathcal H^*(t)
 \le C(A)n\log(e+n)+\frac{\log t}{n+2},  
\end{equation}
and
\begin{equation}\label{eq:FinalVolume}
    \log\log \Vol(B(x,r))
 \le C(A)n\log(e+n)+\frac{2\log r}{n+3}.
\end{equation}

To estimate the right-hand side of \eqref{eq:FinalDislacement}, pick $n=n(t):=3\vee\lfloor\sqrt{\log(t)/\log\log(t)}\rfloor$.

If $\lfloor\sqrt{\log(t)/\log\log(t)}\rfloor\leq 3$, then $t$ is bounded and \eqref{eq:FinalDislacement} is absorbed by possibly enlarged $C(A)$. Otherwise, 
\begin{equation*}
    Cn\log(e+n)+\frac{\log t}{n+1}\leq C\left[n(t)\log n(t)+\frac{\log t}{n(t)+1}\right]\leq C\sqrt{ \log (t) \log\log (t)}. 
\end{equation*}

The estimate of entropy \eqref{eq:FinalEntropy} and volume \eqref{eq:FinalVolume} are similar. This completes the proof.

\end{proof}

\begin{proof}[Proof of Corollary \ref{cor:discretedisplacement}]
To distinguish the two time parametrizations within the proof, temporarily
write \((Y_k)_{k\ge0}\) for the discrete-time lazy walk. Let
\((N_t)_{t\ge0}\) be an independent Poisson process of rate \(2\).
By definition,
\[
    Y_{N_t}
    \stackrel{\mathrm{law}}{=}
    X_t,
\]
where \((X_t)_{t\ge0}\) is the continuous-time walk generated by
\(L=P-I\).

Fix \(n\ge 2e^e\) and set \(t=n/2\). Then
\[
    N_t\sim\operatorname{Poisson}(n).
\]
Since every lazy step moves by graph distance at most one, for every
integer \(m\ge0\),
\[
    \dist(Y_n,Y_m)\le |n-m|.
\]
Consequently, pathwise,
\[
    \dist(x,Y_n)
    \le
    \dist(x,Y_{N_t})+|N_t-n|.
\]
taking expectations gives
\[
\begin{aligned}
    \mathbb E_x\dist(x,Y_n)^2
    &\le
    2\mathbb E_x\dist(x,Y_{N_t})^2
    +2\mathbb E|N_t-n|^2  \\
    &=
    2\mathbb E_x\dist(x,X_{n/2})^2+2n,
\end{aligned}
\]
because a Poisson random variable of mean \(n\) has variance \(n\).

Applying Theorem~\ref{thm:quasi-intro} at time \(n/2\), we obtain
\[
\begin{aligned}
    \mathbb E_x\dist(x,Y_n)^2
    &\le
    n\exp\!\left[
        C_d\sqrt{
            \log(n/2)\,\log\log(n/2)
        }
    \right]
    +2n \\
    &\le
    n\exp\!\left[
        C'_d\sqrt{\log n\,\log\log n}
    \right]
\end{aligned}
\]
for some \(C'_d<\infty\).

Finally, the remaining range \(e^e\le n<2e^e\) is absorbed by enlarging
\(C'_d\), using the elementary bound
\[
    \dist(x,Y_n)\le n.
\]

\end{proof}

\section{Diffusive displacement implies doubling}\label{sec:deficit}
In following section, we further discuss the method we apply. As an application, we prove that diffusive displacement implies volume doubling for reversible Markov chain under $\mathrm{LH}(A)$.

\begin{proposition}\label{prop:ball-deficit}
Let \((M,d)\) be a metric space with measurable balls, let
\(K\) be a Markov kernel on \(M\), and let
\(m_0,m_1\) be \(\sigma\)-finite measures satisfying
\begin{equation}\label{eq:transported-measure}
        m_1(A)=m_0K(A):=\int_M K(y,A)\,m_0(dy)
        \qquad\text{for every measurable }A.
\end{equation}
Fix \(x\in M\) and \(0<r\le R\).  Suppose 
\[
        K(x,B(x,r))>0
\]
and suppose
\begin{equation}\label{eq:ball-deficit-entropy}
        \sup_{y\in B(x,R)}
        D(K(x,\cdot)\|K(y,\cdot))\le H<\infty.
\end{equation}
Then
\begin{equation}\label{eq:ball-deficit-conclusion}
        m_0(B(x,R))
        \le
        \exp\left\{\frac{H+\log2}{K(x,B(x,r))}\right\}
        m_1(B(x,r)).
\end{equation}
\end{proposition}

\begin{proof}
Write
\[
        a:=K(x,B(x,r)),
        \qquad
        b_y:=K(y,B(x,r)), \qquad  y\in B(x,R).
\]

Apply  Proposition \ref{prop:data-processing} to
the measurable map \(z\mapsto\1_{B(x,r)}(z)\):
\begin{equation}\label{eq:bernoulli-deficit}
\begin{aligned}
         D(K(x,\cdot)\|K(y,\cdot))
        &\ge D\!\left((\1_{B(x,r)})_\#K(x,\cdot)
        \,\middle\|\,(\1_{B(x,r)})_\#K(y,\cdot)\right)\\
        &=a\log\frac{a}{b_y}
        +(1-a)\log\frac{1-a}{1-b_y}                 \\
        &=a\log\frac1{b_y}
          +(1-a)\log\frac1{1-b_y}
          -h(a)                                      \\
        &\ge a\log\frac1{b_y}-\log2,
\end{aligned}
\end{equation}
where
\(h(a)=-a\log a-(1-a)\log(1-a)\le\log2\).  Assumption
\eqref{eq:ball-deficit-entropy}
therefore gives
\begin{equation}\label{eq:ball-probability-lower}
        K(y,B(x,r))
        \ge
        \exp\left\{-\frac{H+\log2}{K(x,B(x,r))}\right\},
        \qquad y\in B(x,R).
\end{equation}
In particular, all these probabilities are positive.  Integrating
\eqref{eq:ball-probability-lower}, using
\eqref{eq:transported-measure}, and applying Tonelli's theorem yields
\[
\begin{aligned}
        m_1(B(x,r))
        &=\int_M K(y,A_r)\,m_0(dy) \\
        &\ge
        \exp\left\{-\frac{H+\log2}{K(x,B(x,r))}\right\}
        m_0(B(x,R)),
\end{aligned}
\]
which is \eqref{eq:ball-deficit-conclusion}.  
\end{proof}

\begin{proof}[Proof of Theorem \ref{prop:retention-doubling}]

 Fix $x\in V$ and $r\ge\left(2K\right)^{1/\alpha}$, and set
$$t=\left(\frac{r}{\left(2K\right)^{1/\alpha}}\right)^2\ge1.$$ Markov inequality and \eqref{eq:AlphaDisplacement} gives
\begin{equation}\label{eq:retention-inner-ball}
 P_t(x,B(x,r))\ge  \Prob_x\!\left(\dist(x,X_t)\le
 \left(2K\right)^{1/\alpha}\sqrt t\right)\ge \frac{1}{2}.
\end{equation}
For every $y\in B(x,2r)$, $\mathrm{LH}(A)$ yields
\begin{equation*}
 D(P_t(x,\cdot)\|P_t(y,\cdot))
 \le A(4\left(2K\right)^{2/\alpha}+2\left(2K\right)^{1/\alpha})=:H_0.
\end{equation*}
Detailed balance implies $mP_t=m$. Proposition \ref{prop:ball-deficit}, with inner radius $r$ and outer radius
$2r$, gives
\[
 \Vol(B(x,2r))
 \le
 \exp\left\{2{H_0+2\log2}\right\}\Vol(B(x,r)).
\]
For $0 \le r<\left(2K\right)^{1/\alpha}$,  Proposition \ref{prop:LH-degree} gives a uniform bound
depending only on $A,K,\alpha$.  Combining the two ranges proves doubling.

\end{proof}

\begin{remark}
    In Riemannian geometry, if (M, g) is a complete Riemannian manifold with $Ric\geq 0$, Laplacian comparison theorem with model spaces gives volume doubling  and diffusive displacement. However, in general, a reversible Markov chain with non-negative Ollivier-Ricci curvature may not exhibit a diffusive displacement.
    
    Indeed, let $0<p<q,~p+q=1$ and the biased birth--death chain $P$ on
$\mathbb Z$ with $$P(n,n+1)=p, \qquad P(n,n-1)=q, \text{ for every }  n\in \mathbb Z.$$ Then, $P$ is reversible with respect to
$m(k)=(p/q)^k$ and has non-negative  Ollivier--Ricci curvature. By Theorem \ref{matrixharnack}, it satisfies $LH(A)$ for some large $A$.
Nevertheless,
\[
 \mathbb E_x d(x,X_t)^2
 =
 t+(p-q)^2t^2.
\]
Hence its displacement
from the initial point is not diffusive when $p\neq q$.   Thus, in
the weighted setting, a diffusive-displacement hypothesis is an assumption rather than a consequence of nonnegative Ollivier--Ricci curvature.
\end{remark}

\section{Ollivier curvature implies log-Harnack}
\label{sec:curvature-LH}

We now prove the log-Harnack inequality under non-negative Ollivier-Ricci curvature. We first construct a reference Markovian coupling and then add an extra transition rate that decreases the distance between its two coordinates. Since the additional jump leaves the first coordinate unchanged, the law of the entire first-coordinate path is preserved, while the change of measure incurs an entropy cost. This is a discrete analogue of Wang's coupling-by-change-of-measure method for log-Harnack inequalities. We refer the reader to~\cite{coupling} for a detailed account of this method.

In this final section, we recall that \(P\) denotes an arbitrary irreducible Markov kernel
on a countable state space \(V\).  The one-state case is immediate, so we
assume \(|V|\ge2\).  Put
\begin{equation*}
        L:=P-I,
        \qquad
        P_t:=e^{tL},
        \qquad
        \Pmin:=
        \inf_{\substack{x\ne y\\P(x,y)>0}}P(x,y).
\end{equation*}
We call  \((P_t)_{t\ge0}\)  its heat semigroup.  
We assume $P$ is \textit{weak-reversible}:
\begin{equation*}
        P(x,y)>0\quad\Longleftrightarrow\quad P(y,x)>0
\end{equation*}
and define the transition distance
\begin{equation*}
        \dist(x,y):=\min\{k\ge0:P^k(x,y)>0\}.
\end{equation*}
For probability measures \(\mu,\nu\), write
\[
        W_1(\mu,\nu)
        :=\inf_{\pi}\sum_{u,v}\dist(u,v)\pi(u,v),
\]
where the infimum runs over their couplings. 
It is convenient to define the lazy kernel $\widehat P:=\frac12(I+P)$ .
Finally, we say that $P$ has non-negative Ollivier--Ricci curvature if it is a contraction under $W_1$, that is
if
\[
    W_1\bigl(\widehat P(x,\cdot),\widehat P(y,\cdot)\bigr)
    \leq 1,
    \qquad\text{whenever }x\sim y.
\]
By picking a geodesic along $x$ and $y$,  we have 
\[
    W_1\bigl(\widehat P(x,\cdot),\widehat P(y,\cdot)\bigr)
    \leq \dist(x,y).
\]
The aim of this section is to prove the following theorem.
\begin{theorem}\label{matrixharnack}
    If  $P$ is a weak-reversible Markov kernel  with non-negative Ollivier-Ricci curvature and $\Pmin>0$, then its heat semigroup satisfies the following log-Harnack inequality:
        \begin{equation*}
D\bigl(P_T(x,\cdot)\|P_T(y,\cdot)\bigr)
\le
\frac{32}{P_{\min}}
\left(
    \frac{\dist(x,y)^2}{T}
    +
    \frac{\dist(x,y)}{\sqrt T}
\right).
\end{equation*}
\end{theorem}

Before proving the theorem, we shall do some preparations. The following estimate is classic and we present here for completeness. Comparisons between different $\alpha$-laziness parameters have been widely studied; see, e.g., \cite{Riccicur,Olliveridleness}.
Define
\begin{equation*}
 K:=\frac{1}{2}\left(I+\widehat P\right)=I+\frac14L.
\end{equation*}

\begin{lemma}
\label{lem:K-contraction-compact}
For all $a,b\in V$,
\begin{equation*}
 W_1(K(a,\cdot),K(b,\cdot))\le\dist(a,b).
\end{equation*}
\end{lemma}

\begin{proof}
By convexity of $W_1$,
$$W_1(K(a,\cdot),K(b,\cdot))\le\frac{1}{2}W_1(\delta_a,\delta_b)+\frac{1}{2}W_1(\widehat P(a,\cdot),\widehat P(b,\cdot))\le \dist(a,b).$$
\end{proof}

For every ordered pair $a\ne b$, choose a support-neighbor
$b^-=b^-(a,b)$ of $b$ satisfying
\begin{equation*}
 \dist(a,b^-)=\dist(a,b)-1.
\end{equation*}
We call $(a,b)\to(a,b^-)$ the \emph{geodesic drift}.  Recall that for finitely supported $\mu,\nu$, the $$W_1(\mu,\nu)
        :=\inf_{\pi}\sum_{u,v}\dist(u,v)\pi(u,v)$$can be attained by minimizers, called optimal couplings. Following  \cite{BD2007,M2023OR,MuenchSalez2022}, we choose an optimal coupling so it contains a "decent" probability to the $(a,b^-)$.
It  creates a controllable transition whose drift does not alter the
first marginal.

\begin{lemma}\label{lem:forced-gate}
For every $a\ne b$, there is a coupling $\pi_{a,b}$ of
$K(a,\cdot)$ and $K(b,\cdot)$ such that
\begin{equation}\label{eq:forced-gate-compact}
 \sum_{u,v}\dist(u,v)\pi_{a,b}(u,v)\le\dist(a,b),
 \qquad
 \pi_{a,b}(a,b^-)=K(b,b^-)=\frac14P(b,b^-).
\end{equation}
\end{lemma}

\begin{proof}
By Lemma \ref{lem:K-contraction-compact} and compactness, there is a coupling \(\pi\) of $K(a,\cdot)$ and $K(b,\cdot)$ satisfying: 
$$\sum_{u,v}\dist(u,v)\pi(u,v)\le\dist(a,b).$$
If \(\pi(a,b^-)=K(b,b^-)\), there is nothing to prove.  Otherwise set
\[
\frac{1}{4} \ge s:=K(b,b^-)-\pi(a,b^-)>0.
\]
The second marginal gives
\[
        \sum_{u\neq a}\pi(u,b^-)=s.
\]
The first marginal mass of row \(a\) is \(K(a,a)\ge3/4\), while the second
marginal mass away from \(b\) is \(K(b,V\setminus\{b\})\le1/4\).  Hence
every coupling has
\[
        \pi(a,b)\ge\frac34-\frac14=\frac12.
\]
In particular \(\pi(a,b)\ge s\).  
Define the modified transport plan $\widetilde\pi$ explicitly by
\[
\widetilde\pi(u,v):=
\begin{cases}
\pi(a,b^-)+s,
    &(u,v)=(a,b^-),\\
\pi(a,b)-s,
    &(u,v)=(a,b),\\
0,
    &u\neq a,\ v=b^-,\\
\pi(u,b)+\pi(u,b^-),
    &u\neq a,\ v=b,\\
\pi(u,v),
    &\text{otherwise}.
\end{cases}
\]
By construction, $\widetilde\pi$ is non-negative and has the same marginals as $\pi$. Thus \(\widetilde\pi\) is again a
coupling of \(K(a,\cdot)\) and \(K(b,\cdot)\).  Moreover,
\[
        \widetilde\pi(a,b^-)
        =\pi(a,b^-)+s
        =K(b,b^-).
\]
Finally, the change in transport cost is
\[
\begin{aligned}
&\sum_{u,v}\dist(u,v)
  \{\widetilde\pi(u,v)-\pi(u,v)\}\\
&\quad=
\sum_{u\neq a}\pi(u,b^-)
\bigl\{
\dist(a,b^-)+\dist(u,b)
-\dist(a,b)-\dist(u,b^-)
\bigr\}\\
&\quad=
\sum_{u\neq a}\pi(u,b^-)
\bigl\{-1+\dist(u,b)-\dist(u,b^-)\bigr\}
\le0,
\end{aligned}
\]
because \(b\sim b^-\) implies
\(\dist(u,b)\le\dist(u,b^-)+1\).  Taking
\(\pi_{a,b}:=\widetilde\pi\) proves \eqref{eq:forced-gate-compact}.
\end{proof}
Now, we are going to ready to construct the reference measure and change the reference measure by a geodesic drift.
\begin{proposition}[Reference coupling]\label{prop:reference-coupling}
There is a continuous-time Markovian coupling of two copies of the
\(L=P-I\) chain, sticky on the diagonal, with generator \(\mathcal A_0\) on
\(V\times V\), such that
for \(\dist(\cdot,\cdot)\),
\begin{equation*}
        \mathcal A_0\dist(a,b)\le0,\qquad  \mathcal A_0\dist^2(a,b)\leq 16
        \qquad a,b\in V,
\end{equation*}
and whenever \(a\neq b\), the reference rate of the geodesic drift is
\[
        (a,b)\longrightarrow (a,b^-)
\]
has rate
\begin{equation*}
 \kappa(a,b)
 :=4\pi_{a,b}(a,b^-)
 =P(b,b^-)
 \ge P_{\min}.
\end{equation*}
\end{proposition}

\begin{proof}
Define a Markov kernel \(\Pi_0\) on the product space \(V\times V\) as
follows.  If \(a\neq b\), let
\[
        \Pi_0((a,b),(u,v)):=\pi_{a,b}(u,v),
\]
where \(\pi_{a,b}\) is the coupling from Lemma \ref{lem:forced-gate}.  On the
diagonal, use the synchronous \(K\)-coupling:
\[
        \Pi_0((a,a),(u,v))
        :=
        K(a,u)\1_{\{u=v\}}.
\]
Now set
\[
        \mathcal A_0:=4(\Pi_0-I).
\]
 In particular,
\[
       \mathcal A_0F(a,b)
        =4\sum_{u,v}\{F(u,v)-F(a,b)\}\pi_{a,b}(u,v),
        \qquad a\neq b.
\]
Thus \(\mathcal A_0\) is a generator on the product state space.  The factor \(4\) turns the lazy marginal kernel
back into the original generator because \(4(K-I)=L\).

The diagonal is sticky as a set: from \((a,a)\) the process
moves only to points \((u,u)\).  Moreover, for \(a\neq b\),
\begin{equation}\label{eq:lapA0}
        \mathcal A_0\dist(a,b)
        =
        4\sum_{u,v}\{\dist(u,v)-\dist(a,b)\}\pi_{a,b}(u,v)
        \le0.
\end{equation}
by Lemma \ref{lem:K-contraction-compact}, while on the diagonal \(A_0\dist(a,a)=0\).
Since $ \operatorname{supp}K(z,\cdot)\subseteq B(z,1)$, every $(u,v)\in\operatorname{supp}\pi_{a,b}$ satisfies$$\bigl|\dist(u,v)-\dist(a,b)\bigr|\le2.$$
Together with \eqref{eq:lapA0}, this gives
\begin{equation}
\mathcal A_0\dist^2(a,b)= 2\dist(a,b)\mathcal A_0\dist(a,b)+4\sum_{u,v}
        \{\dist(u,v)-\dist(a,b)\}^2\pi_{a,b}(u,v)
        \le 16.
\end{equation}

Finally, \eqref{lem:forced-gate}
gives that the off-diagonal drift \((a,b)\to(a,b^-)\) has rate
\[
        \kappa(a,b):=4\pi_{a,b}(a,b^-)
        =P(b,b^-)\ge\Pmin.
\]
\end{proof}

\subsection*{The entropy formula}
The following is an application of Girsanov's theorem; see \cite[Part~I.5]{Point2021} and
\cite[Chapter~III]{JacodShiryaev2003}. 
Let \(Z_t=(X_t,Y_t)\) be a coordinate process, \(\mathcal F_t:=\sigma(Z_s:0\le s\le t),
\) and let
\(
\Delta=\{(a,a):a\in V\}.
\) For \(z=(a,b)\notin\Delta\), set
\[
g(z):=(a,b^-),
\]
and define
\[
N_t^\mathrm g
:=
\sum_{0<s\le t}
\mathbf{1}_{\{Z_s=g(Z_{s-})\}}
\]which counts the geodesic-drift jumps. 

Suppose that, off the diagonal, only the rate of geodesic drift is changed from
$\kappa(Z_{s-})$ to $\kappa(Z_{s-})+\lambda_s(Z_{s-})$, where
$\lambda_s\ge0$ is predictable. Equivalently, we change the reference measure to a time-inhomogeneous Markov chain, in  which  the controlled
time-dependent generator is
\begin{equation*}
        \mathcal A_s^\lambda F(a,b)
        :=
        \mathcal A_0F(a,b)
        +
        \lambda_s(a,b)\{F(a,b^-)-F(a,b)\},
        \qquad a\neq b,
\end{equation*}
and \(\mathcal A_s^\lambda F(a,a):=\mathcal A_0F(a,a)\) on the diagonal.

Let \(\mathbf{P}^0\) denote the law under which the coordinate process
\(Z\) has generator \(\mathcal{A}_0\), and let \(\mathbf{Q}\) denote the
law under which \(Z\) has time-dependent generator
\((\mathcal{A}_s^\lambda)_{s<T}\). 

For $\kappa>0$ and $\lambda\ge0$, define
\begin{equation*}
 \Psi_\kappa(\lambda)
 :=(\kappa+\lambda)
 \log\frac{\kappa+\lambda}{\kappa}-\lambda.
\end{equation*}

Fix $t<T$ and stop the process at a stopping time $\sigma$ on which all
rates are bounded.  Then the path likelihood is
\begin{equation}\label{eq:one-channel-likelihood}
 \log\frac{d\mathbf Q_{t,\sigma}}
          {d\mathbf P^0_{t,\sigma}}
 =
 \int_0^{t\wedge\sigma}
 \log\frac{\kappa(Z_{s-})+\lambda_s(Z_{s-})}
          {\kappa(Z_{s-})}\,dN_s^{\mathrm g}
 -
 \int_0^{t\wedge\sigma}
 \lambda_s(Z_{s-})
 \mathbf 1_{\{Z_{s-}\notin\Delta\}}\,ds.
\end{equation}
Under $\mathbf Q$, the  process
\[
 N_{t\wedge\sigma}^{\mathrm g}
 -
 \int_0^{t\wedge\sigma}
 \{\kappa(Z_{s-})+\lambda_s(Z_{s-})\}
 \mathbf 1_{\{Z_{s-}\notin\Delta\}}\,ds
\]
is a martingale.  Taking $\mathbf Q$-expectations in
\eqref{eq:one-channel-likelihood} therefore yields
\begin{equation}\label{eq:one-channel-entropy}
 D(\mathbf Q_{t,\sigma}\|\mathbf P^0_{t,\sigma})
 =
 \mathbb E_{\mathbf Q}\int_0^{t\wedge\sigma}
 \Psi_{\kappa(Z_{s-})}
 \bigl(\lambda_s(Z_{s-})\bigr)
 \mathbf 1_{\{Z_{s-}\notin\Delta\}}\,ds.
\end{equation}

We will also use the stopped time-inhomogeneous Dynkin formula
\begin{equation}\label{eq:localized-dynkin}
 \mathbb E_{\mathbf Q}
 V_{t\wedge\sigma}(Z_{t\wedge\sigma})
 =
 V_0(Z_0)
 +
 \mathbb E_{\mathbf Q}\int_0^{t\wedge\sigma}
 (\partial_s+A_s^\lambda)V_s(Z_{s-})\,ds.
\end{equation}
 Indeed, adjoining time to the state gives the space--time process
$\widehat Z_s=(s,Z_s)$, whose generator acts on
$\widehat V(s,z)=V_s(z)$ as
\[
 \widehat A\widehat V(s,z)
 =
 \partial_sV_s(z)+A_s^\lambda V_s(z).
\]
Thus \eqref{eq:localized-dynkin} is the classic Dynkin formula; see~\cite[Theorem~9.21]{Klebaner}. In the
application below we take
\[
 \sigma=\sigma_n:=\inf\{s:\dist(Z_s)\ge n\},
\]
first let $n\to\infty$ for fixed $t<T$, and only afterwards let $t\uparrow T$.

Now, we introduce the following lemma that would be used in the following. This is the order-one case of the continuity theorem for R\'enyi
divergences along increasing sigma-fields in 
\cite[Theorem~21]{vanErvenHarremoes2014}. In the notation of that theorem, \(D_1\) is the
Kullback--Leibler divergence.
\begin{lemma}[Continuity along increasing sigma-fields]\label{prop:increasing-entropy}
Let $\mathcal F_n\uparrow\mathcal F_\infty$, where
\[
 \mathcal F_\infty=\sigma\!\left(\bigcup_{n\geq1}\mathcal F_n\right).
\]
For probability laws $\mathbf Q,\mathbf P^0$ on the same measurable space,
\begin{equation*}
 D\bigl(\mathbf Q|_{\mathcal F_\infty}\,\|\,\mathbf P^0|_{\mathcal F_\infty}\bigr)
 =\sup_n D\bigl(\mathbf Q|_{\mathcal F_n}\,\|\,\mathbf P^0|_{\mathcal F_n}\bigr).
\end{equation*}
Consequently, if $\mathcal F_{T-}:=\sigma(\bigcup_{t<T}\mathcal F_t)$, then
\begin{equation*}
 D\bigl(\mathbf Q|_{\mathcal F_{T-}}\,\|\,\mathbf P^0|_{\mathcal F_{T-}}\bigr)
 =\sup_{t<T}D\bigl(\mathbf Q|_{\mathcal F_t}\,\|\,\mathbf P^0|_{\mathcal F_t}\bigr).
\end{equation*}
\end{lemma}

\begin{theorem}    \label{prop:controlled-coalescence}
Let $\mathbf P^0$ be the reference coupling above, started from $(x,y)$.  let $T>0$  and Set
\(
    \mathcal F_{T-}
    :=
    \sigma\!\left(\bigcup_{t<T}\mathcal F_t\right).
\)  There is a time-inhomogeneous law $\mathbf Q$ on pair paths
such that:
\begin{enumerate}[label=\textup{(\roman*)},leftmargin=2.2em]
\item the entire $X$-path has under $\mathbf Q$ the original $P_t$-law
started from $x$;
\item $\mathbf Q(\tau<T)=1$, where
$\tau:=\inf\{s:X_s=Y_s\}$;
\item
$
 D\!\left(
    \mathbf Q|_{\mathcal F_{T-}}
    \,\middle\|\,
    \mathbf P^0|_{\mathcal F_{T-}}
\right)
\leq 
\frac{32}{P_{\min}}
\left(
    \frac{r^2}{T}
    +
    \frac{r}{\sqrt T}
\right),~
 r=\dist(x,y).
$
\end{enumerate}
\end{theorem}

\begin{proof}
Write  $r=\dist(x,y)$  and let $a\neq b$ in $V$; the case $r=0$ is immediate.  For $s<T$, put
$u=T-s$ and
\begin{equation*}
 V_s(a,b)
 :=B\left(\frac{\dist(a,b)^2}{u}
          +\frac{\dist(a,b)}{\sqrt u}\right),
 \qquad B:=\frac{32}{P_{\min}}.
\end{equation*}
At an off-diagonal state with $m=\dist(a,b)\geq 1$, define the
\begin{equation*}
 H_s(a,b)
 :=V_s(a,b)-V_s(a,b^-)
 =B\left(\frac{2\dist(a,b)-1}{u}+\frac1{\sqrt u}\right)\geq 0.
\end{equation*}
Increase only the geodesic drift rate, choosing
\begin{equation*}
 \lambda_s(a,b)
 :=\kappa(a,b)\bigl(e^{H_s(a,b)}-1\bigr).
\end{equation*}
and $\lambda_s(a,a):=0$. The geodesic drift jump leaves the first
coordinate fixed; hence the entire $X$-path marginal remains the original
chain with generator $L$.

The controlled process is non-explosive on $[0,T)$.  If $N_t^0$ counts the
rate-four reference rings and $N_t^\lambda$ the added geodesic drift jumps, then every added jump decreases $\dist(X_s,Y_s)$ by one, whereas each reference ring increases it by at most two. Hence, pathwise,
\begin{equation}\label{eq:nonexplosion-compact}
 N_t^\lambda\le r+2N_t^0,
 \qquad t<T.
\end{equation}
in this construction. Moreover, the added jumps leave the first coordinate unchanged, while the reference part has first marginal generator \(L\).

Now, at the state with \(\dist(a,b)=m\ge1\) , direct calculation yields
\[
        \partial_sV_s(a,b)
        =B\left(\frac{m^2}{u^2}+\frac{m}{2u^{3/2}}\right),
\]
while
\[
        \mathcal A_0V_s(a,b)
        =B\left(\frac{\mathcal A_0\dist^2(a,b)}u
        +\frac{\mathcal A_0\dist(a,b)}{\sqrt u}\right)
        \le\frac{16B}{u}.
\]
Adding these two inequalities gives
\begin{equation}\label{eq:uncontrolled-lyapunov-drift}
        (\partial_s+A_0)V_s(a,b)
        \le
        B\left(
        \frac{m^2}{u^2}
        +
        \frac{m}{2u^{3/2}}
        +
        \frac{16}{u}
        \right).
\end{equation}

The  contribution to the drift of \(V\) plus the entropy rate is
\begin{equation}\label{eq:phi-cancellation}
\begin{aligned}
        &\lambda_s(a,b)(V_s(a,b^-)-V_s(a,b))
        +\Psi_{\kappa(a,b)}(\lambda_s(a,b))                 \\
        &\qquad =
        \kappa(a,b)\{1+H_s(a,b)-e^{H_s(a,b)}\}
        \le
        -\frac{\Pmin}{2}H_s(a,b)^2.
\end{aligned}
\end{equation}
Indeed, \(\kappa+\lambda_s=\kappa e^{H_s}\), so substitution into the definition
of \(\Psi\) gives the equality.  We then used \(\kappa(a,b)\ge\Pmin\) and
\(1+h-e^h\le-h^2/2\) for \(h\ge0\).  The cancellation in the middle line is
the reason for the choice
\(\lambda_s=\kappa(e^{H_s}-1)\).
Since \(m\ge1\),
\begin{equation}\label{eq:drop-square-lower}
        H_s(a,b)^2
        \ge
        B^2
        \left(
        \frac{m^2}{u^2}
        +
        \frac{2m}{u^{3/2}}
        +
        \frac1u
        \right).
\end{equation}
Combining \eqref{eq:uncontrolled-lyapunov-drift}--\eqref{eq:drop-square-lower} and note that 
$$ (\partial_s+A_s^\lambda)V_s(a,b)= (\partial_s+A_0)V_s(a,b)+\lambda_s(a,b)(V_s(a,b^-)-V_s(a,b)),$$
we yield
\begin{equation}\label{eq:lyapunov-entropy-drift}
        (\partial_s+A_s^\lambda)V_s(a,b)
        +
        \Psi_{\kappa(a,b)}(\lambda_s(a,b))
        \le0
\end{equation}
for all \(a\neq b\), and both terms vanish on the diagonal.

Fix $t<T$ and let
\[
 \sigma_n:=\inf\{s\geq0:\dist(X_s,Y_s)\geq n\},
 \qquad
 \theta_n:=t\wedge\sigma_n.
\]
Up to $\theta_n$, the distance, the controlled rates, and all quantities entering Dynkin's formula are bounded.  Let
\[
 \mathcal G_n:=\mathcal F_{\theta_n}.
\]
Combining \eqref{eq:one-channel-entropy} and Dynkin's formula \eqref{eq:localized-dynkin}  with \eqref{eq:lyapunov-entropy-drift} gives
\begin{equation}\label{eq:localLyaponov}
 \E_{\mathbf Q}V_{\theta_n}(X_{\theta_n},Y_{\theta_n})
 +D\bigl(\mathbf Q|_{\mathcal G_n}\,\|\,\mathbf P^0|_{\mathcal G_n}\bigr)
 \leq V_0(x,y).
\end{equation}
By \eqref{eq:nonexplosion-compact}, $\sup_{s\leq t}\dist(X_s,Y_s)<\infty$ almost surely under $\mathbf Q$; the same is immediate under $\mathbf P^0$.  Hence $\sigma_n>t$ eventually under both laws.  The sigma-fields $\mathcal G_n$ increase to $\mathcal F_t$.  Lemma~\ref{prop:increasing-entropy}, Fatou's lemma, and \eqref{eq:localLyaponov} therefore yield
\begin{equation}\label{eq:lyapunov-entropy-budget}
 \E_{\mathbf Q}V_t(X_t,Y_t)
 +D\bigl(\mathbf Q|_{\mathcal F_t}\,\|\,\mathbf P^0|_{\mathcal F_t}\bigr)
 \leq V_0(x,y),
 \qquad t<T.
\end{equation}

On $\{\tau>t\}$, the integer-valued distance is at least one, and hence
\[
 V_t(X_t,Y_t)\ge\frac{B}{\sqrt{T-t}}.
\]
Thus
\[
 \mathbf Q(\tau>t)
 \le\frac{V_0(x,y)}B\sqrt{T-t}\longrightarrow0
 \qquad(t\uparrow T),
\]
so $\mathbf Q(\tau<T)=1$.

Recall
\[
    \mathcal F_{T-}
    :=
    \sigma\!\left(\bigcup_{t<T}\mathcal F_t\right).
\]
and choose any sequence \(t_n\uparrow T\). 
Lemma~\ref{prop:increasing-entropy} and \eqref{eq:lyapunov-entropy-budget} give
\[
\begin{aligned}
D\!\left(
    \mathbf Q|_{\mathcal F_{T-}}
    \,\middle\|\,
    \mathbf P^0|_{\mathcal F_{T-}}
\right)
&=
\lim_{n\to\infty}
D\!\left(
    \mathbf Q|_{\mathcal F_{t_n}}
    \,\middle\|\,
    \mathbf P^0|_{\mathcal F_{t_n}}
\right)                                                     \\
&\le
V_0(x,y)
=
\frac{32}{P_{\min}}
\left(
    \frac{r^2}{T}
    +
    \frac{r}{\sqrt T}
\right),
\qquad r=\dist(x,y).
\end{aligned}
\]

\end{proof}
\begin{proof}[Proof of Theorem  \ref{matrixharnack}]
    Since the paths have left limits, the  map
\[
    e_{T-}^{Y}:\omega\longmapsto Y_{T-}(\omega)
\]
is \(\mathcal F_{T-}\)-measurable.  Under the reference law
\(\mathbf P^0\), the second coordinate is the original chain started
from \(y\).  Since an ordinary continuous-time chain has no jump at the
fixed deterministic time \(T\),
\[
    (e_{T-}^{Y})_{\#}\mathbf P^0
    =
    P_T(y,\cdot).
\]
Under \(\mathbf Q\), the coupling time satisfies \(\tau<T\) almost
surely, and the diagonal is absorbing.  Hence
\[
    Y_{T-}=X_{T-}
    \qquad \mathbf Q\text{-a.s.}
\]
Moreover, the complete \(X\)-path marginal under \(\mathbf Q\) is the
law of the original chain started from \(x\), and therefore
\[
    (e_{T-}^{Y})_{\#}\mathbf Q
    =
    \Law_{\mathbf Q}(X_{T-})
    =
    P_T(x,\cdot).
\]
Proposition \eqref{prop:data-processing} applied to \(e_{T-}^{Y}\) now yields
\[
\begin{aligned}
D\bigl(P_T(x,\cdot)\|P_T(y,\cdot)\bigr)
&\le
D\!\left(
    \mathbf Q|_{\mathcal F_{T-}}
    \,\middle\|\,
    \mathbf P^0|_{\mathcal F_{T-}}
\right)                                                     \\
&\le
\frac{32}{P_{\min}}
\left(
    \frac{\dist(x,y)^2}{T}
    +
    \frac{\dist(x,y)}{\sqrt T}
\right).
\end{aligned}
\]
Consequently, for every bounded $f:V\to(0,\infty)$,
\begin{equation*}\label{eq:general-log-harnack}
 P_T\log f(x)
 \le\log P_Tf(y)
 +\frac{32}{P_{\min}}
 \left(\frac{\dist(x,y)^2}{T}
       +\frac{\dist(x,y)}{\sqrt T}\right).
\end{equation*}
\end{proof}
\begin{proof}[Proof of Theorem~\ref{thm:entropy-intro}]
    For a graph with bounded degree $d$,  we have $P_{\min}\geq 1/d$. Theorem \ref{matrixharnack} applies.
\end{proof}

\begin{proof}[Proof of Theorem~\ref{thm:quasi-intro}]
    Apply Theorem \ref{thm:entropy-intro} and Theorem  \ref{thm:quasi-intro-lh}.
    \end{proof}

\section*{Acknowledgments}
We thank Bang-Xian Han, Xueping Huang, Liming Yin, and Zhuo-Nan Zhu for helpful comments and
discussion.

\bibliographystyle{plain}
\bibliography{displacement}

\end{document}